\documentclass[a4,12pt]{article}
\usepackage{amssymb,amsmath,epsfig,amscd,eucal,psfrag,amsthm,enumerate,tikz-cd}
\usepackage[T1]{fontenc}
\usepackage[latin9]{inputenc}
\usepackage{graphicx,color}

\pdfoutput=1

\newtheorem{theorem}{Theorem}

\newtheorem{corollary}[theorem]{Corollary}

\newtheorem{question}[theorem]{Question}

\makeatletter
\newtheorem*{rep@theorem}{\rep@title}
\newcommand{\newreptheorem}[2]{%
\newenvironment{rep#1}[1]{%
 \def\rep@title{#2 \ref{##1}}%
 \begin{rep@theorem}}%
 {\end{rep@theorem}}}
\makeatother

\newreptheorem{theorem}{Theorem}
\newreptheorem{cor}{Corollary}

\theoremstyle{definition}

\theoremstyle{remark}
\newtheorem{remark}[theorem]{Remark}

\newcommand{\Z}{\mathbb{Z}}

\usepackage{hyperref}

\input xy
\xyoption{all}

\title{On the profinite rigidity of surface groups and surface words}
\author{Henry Wilton}

\newcommand{\Addresses}{{
  \bigskip
  \footnotesize

  \textsc{DPMMS, Centre for Mathematical Sciences, Wilberforce Road, Cambridge, CB3 0WB, UK}\par\nopagebreak
  \textit{E-mail address:} \texttt{h.wilton@maths.cam.ac.uk}

}}

\begin{document}

\maketitle

\begin{abstract}
Surface groups are determined among limit groups by their profinite completions.  As a corollary, the set of surface words in a free group is closed in the profinite topology.
\end{abstract}

Profinite rigidity is the study of the extent to which a group $G$ is determined by its set of finite quotients, or, in more sophisticated terminology, by its profinite completion $\widehat{G}$.  Since $\widehat{G}$ only depends on the image of the natural map $G\to\widehat{G}$, we should assume that $G$ is residually finite, which means that the map is injective.  Remeslennikov asked one of the most provocative questions in this area \cite[Question 15]{noskov_infinite_1979}.

\begin{question}[Remeslennikov]\label{qu: Remeslennikov}
Let $F$ be a free (non-abelian) group, and let $G$ be finitely generated and residually finite. If $\widehat{G}\cong\widehat{F}$, does it follow that $G\cong F$?
\end{question}

In terminology that has recently become popular, Remeslennikov's question asks if free groups are \emph{absolutely profinitely rigid}.   Equally good questions can be formulated by replacing the free group $F$ by other groups of interest, such as surface groups.

Many finitely generated groups are known not to be absolutely profinitely rigid; perhaps the easiest examples are a pair of meta-cyclic groups exhibited by Baumslag \cite{baumslag_residually_1974}. There are some easy examples of absolutely profinitely rigid groups -- for instance, finite groups and finitely generated abelian groups -- but beyond these and their generalisations no finitely generated examples of absolutely profinitely rigid groups were known until the recent work of Bridson, McReynolds, Reid and Spitler \cite{bridson_absolute_2020,bridson_profinite_2020}, who showed that certain Kleinian groups, and even certain triangle groups, are absolutely profinitely rigid.  However, their techniques rely on delicate arithmetic information, and an answer to Question \ref{qu: Remeslennikov} remains a distant prospect, as does its analogue for surface groups.

One way to make progress is to place more restrictive hypotheses on $G$. For instance, Bridson, Conder and Reid showed that Fuchsian groups are profinitely rigid among lattices in Lie groups \cite{bridson_determining_2016}.  We are concerned with a different restriction.

A group $G$ is said to be \emph{fully residually free} if every finite subset $S\subseteq G$ can be mapped injectively by a homomorphism to a free group $F$.  The finitely generated fully residually free groups, called \emph{limit groups} by Sela, form a rich class of groups that are often difficult to distinguish from free groups, and that include all surface groups except for the three non-orientable surfaces with $\chi\geq -1$ \cite{baumslag_residually_1967}.  They came to prominence because of the central role that they play in Sela's solution to Tarski's problems on the elementary theory of free groups (see \cite{sela_diophantine_2001} \emph{et seq.}; see also the work of Kharlampovich--Miasnikov \cite{kharlampovich_irreducible_1998} \emph{et seq.}).

As a consequence of the existence of surface subgroups in non-free limit groups, free groups are profinitely rigid among limit groups \cite[Corollary D]{wilton_essential_2018}.  The main theorem of this note extends that result to surface groups.

\begin{theorem}\label{thm: Main intro thm}
Let $G$ be a limit group, and suppose that $\widehat{G}\cong\widehat{\pi_1\Sigma}$, where $\Sigma$ is a closed surface. Then $G\cong\pi_1\Sigma$.
\end{theorem}
\begin{proof}
Suppose that $G$ is a limit group with $\widehat{G}\cong\widehat{\pi_1\Sigma}$. By \cite[Corollary D]{wilton_essential_2018}, $G$ is not free, so $G$ contains a non-trivial surface subgroup $\pi_1S$ by \cite[Corollary C]{wilton_essential_2018}.   

By \cite{wilton_halls_2008}, there is a retraction
\[
r:K\to \pi_1S\,,
\]
where $K$ is a subgroup of finite index in $G$. Unless $K$ is a surface group, the kernel of $r$ contains an infinite cyclic subgroup $\langle k\rangle$, since limit groups are torsion-free.  Passing to profinite completions, $r$ induces a continuous retraction
\[
\hat{r}:\widehat{K}\to \widehat{\pi_1S}
\]
which contains the infinite pro-cyclic subgroup $\overline{\langle k\rangle}\cong\widehat{\Z}$ in its kernel.  In particular, the supernatural number $2^\infty$ divides the index of $\widehat{\pi_1S}$ in $\widehat{K}$.

This implies that $\widehat{K}$ is not a profinite surface group. Indeed, if it were, then $\widehat{K}$ would be a profinite $PD_2$ group, so by an observation of Serre (\cite[Exercise 5(b) on p.\ 44]{serre_galois_1997}, and  also \cite[Lemma 1.6]{wilton_profinite_2019}), $\widehat{\pi_1S}$ would have trivial second continuous cohomology.   But surface groups are good \cite[Proposition 3.7]{grunewald_cohomological_2008}, so
\[
H^2(\widehat{\pi_1S},\Z/2)\cong H^2(\pi_1S,\Z/2)\cong\Z/2\,, 
\]
which is a contradiction.

Hence $K$ is a surface group, so $G$ is also a surface group by Nielsen realisation, since it is torsion-free. But surface groups are distinguished from each other by their profinite completions, so $G\cong\pi_1\Sigma$.
\end{proof}

Similar questions can also be asked fruitfully about subsets of a finitely generated free group $F$.  The \emph{profinite topology} on a group is generated by the cosets of the subgroups of finite index, and a subset is called \emph{separable} if it is closed in the profinite topology.  The following question is then the analogue of absolute profinite rigidity for subsets of free groups (cf.\ \cite[pp. 93--94]{puder_measure_2015}).

\begin{question}\label{qu: Separability of automorphism orbits}
Let $w\in F$. Is the automorphism orbit $\mathrm{Aut}(F).w$ separable in $F$?
\end{question}

This question appears to be very difficult for any $w\neq 1$.  The case where $w$ is a \emph{primitive} element of $F$ -- i.e.\ an element of a free basis -- was answered affirmatively by Parzanchevski and Puder \cite{puder_measure_2015}, who proved that primitive elements are distinguished by the push-forward measures they induce under random maps to finite symmetric groups.   A different proof was later given by the author \cite[Corollary E]{wilton_essential_2018}, as a corollary of the profinite rigidity of free groups among limit groups.

Surface words are often regarded as the `next simplest' class of words after primitive words.  For $n$ even, the \emph{orientable} surface words are those in the automorphism orbit of the product of commutators $[a_1,a_2]\ldots [a_{n-1},a_n]$. For any $n$, the \emph{non-orientable} surface words are those in the orbit of the product of squares $a_1^2\ldots a_n^2$.   The terminology stems from the fact that, if $n$ is a surface word, then the presentation
\[
F=\langle a_1,\ldots, a_n,b\mid b=w\rangle
\]
arises when one considers $F$ as the fundamental group of a compact surface with a single boundary component $b$.

Magee and Puder studied surface words in \cite{magee_surface_2019}, and showed that they are characterised by the measure that they induce under random homomorphisms to \emph{compact} Lie groups. However, since the Lie groups they consider are sometimes infinite, this does not imply that the set of surface words is separable.   Hanany, Meiri and Puder proved that the automorphism orbit of  the commutator $[a_1,a_2]$ is separable \cite{hanany_orbits_2020}, again by studying the push-forward measure associated to a random map to a symmetric group.

It is an easy corollary of Theorem \ref{thm: Main intro thm} that any automorphism orbit of surface words is separable.

\begin{corollary}\label{cor: Surface words}
The sets of orientable and non-orientable surface words in a free group $F$ are separable.
\end{corollary}
\begin{proof}
A word $w\in F$  is a surface word if and only if the double
\[
D(w)=F*_{\langle w\rangle} F
\]
is isomorphic to  $\pi_1\Sigma$, where $\Sigma$ is a closed hyperbolic surface. On the other hand, $w$ is in the closure of the set of surface words if and only if $\widehat{D(w)}$ is a profinite hyperbolic surface group. Therefore, it suffices to prove that, if $\widehat{D(w)}$ is a profinite hyperbolic surface group, then $D(w)$ is a hyperbolic surface group.

If $w$ were a proper power, say $w=u^n$ for $n>1$, then $D(w)$ would contain a copy of the torus-link group $\Z*_{n\Z}\Z$, and thence a copy of $\mathbb{Z}^2$. Every finite-index subgroup of this $\Z^2$ would be separable in $D(w)$ \cite[Theorem 5.1]{wise_subgroup_2000}, so $\widehat{D(w)}$ would contain a copy of $\widehat{\Z}^2$. But profinite hyperbolic surface groups do not contain $\widehat{\Z}^2$ \cite[Theorem D]{wilton_distinguishing_2017}, so $w$ is not a proper power.\footnote{The argument of this paragraph should also have been included in the proof of \cite[Corollary E]{wilton_essential_2018}.}

Therefore, $w$ is not a proper power, so $D(w)$ is a limit group \cite{baumslag_generalised_1962}, and hence a surface group  by Theorem \ref{thm: Main intro thm}, so $w$ is a surface word. Finally, (non)-orientability is determined by whether or not the corresponding surface is orientable. Since this is detected by the continuous cohomology group
\[
H^2(\widehat{D(w)},\Z/3)\cong H^2(D(w),\Z/3)
\] 
it follows that the subsets of orientable and non-orientable surface words are separable.
\end{proof}

\begin{remark}\label{rem: Tuples}
There is also a notion of a \emph{surface tuple} of words, arising from surfaces with more than one boundary component. The methods of this paper also show that automorphism orbits of surface tuples are separable, but the details are left to the reader.
\end{remark}

\begin{remark}\label{rem: Powers}
By \cite[Theorem 1.7]{hanany_orbits_2020}, Corollary \ref{cor: Surface words} also extends to powers of surface words.
\end{remark}

\subsection*{Acknowledgements}

I am grateful to Doron Puder for pointing out \cite[Theorem 1.7]{hanany_orbits_2020}.

\bibliographystyle{plain}

\begin{thebibliography}{10}

\bibitem{baumslag_residually_1967}
Benjamin Baumslag.
\newblock Residually free groups.
\newblock {\em Proceedings of the London Mathematical Society. Third Series},
  17:402--418, 1967.

\bibitem{baumslag_generalised_1962}
Gilbert Baumslag.
\newblock On generalised free products.
\newblock {\em Math. Z.}, 78:423--438, 1962.

\bibitem{baumslag_residually_1974}
Gilbert Baumslag.
\newblock Residually finite groups with the same finite images.
\newblock {\em Compositio Math.}, 29:249--252, 1974.

\bibitem{bridson_determining_2016}
M.~R. Bridson, M.~D.~E. Conder, and A.~W. Reid.
\newblock Determining {F}uchsian groups by their finite quotients.
\newblock {\em Israel J. Math.}, 214(1):1--41, 2016.

\bibitem{bridson_absolute_2020}
M.~R. Bridson, D.~B. McReynolds, A.~W. Reid, and R.~Spitler.
\newblock Absolute profinite rigidity and hyperbolic geometry.
\newblock {\em Annals of Math.}, 192(3), 2020.

\bibitem{bridson_profinite_2020}
M.~R. Bridson, D.~B. McReynolds, A.~W. Reid, and R.~Spitler.
\newblock On the profinite rigidity of triangle groups.
\newblock \texttt{arXiv:2004.07137}, 2020.

\bibitem{grunewald_cohomological_2008}
F.~Grunewald, A.~Jaikin-Zapirain, and P.~A. Zalesskii.
\newblock Cohomological goodness and the profinite completion of {B}ianchi
  groups.
\newblock {\em Duke Math. J.}, 144(1):53--72, 2008.

\bibitem{hanany_orbits_2020}
Liam Hanany, Chen Meiri, and Doron Puder.
\newblock Some orbits of free words that are determined by measures on finite
  groups.
\newblock {\em J. Algebra}, 555:305--324, 2020.

\bibitem{kharlampovich_irreducible_1998}
Olga Kharlampovich and Alexei Myasnikov.
\newblock Irreducible affine varieties over a free group. {{I}.}
  {{I}rreducibility} of quadratic equations and {{N}ullstellensatz}.
\newblock {\em Journal of Algebra}, 200(2):472--516, 1998.

\bibitem{magee_surface_2019}
Michael Magee and Doron Puder.
\newblock Surface words are determined by word measures on groups.
\newblock \texttt{arXiv:1902.04873}, 2019.

\bibitem{noskov_infinite_1979}
G.~A. Noskov, V.~N. Remeslennikov, and V.~A. Roman'kov.
\newblock Infinite groups.
\newblock In {\em Algebra. {T}opology. {G}eometry, {V}ol. 17 ({R}ussian)},
  pages 65--157, 308. Akad. Nauk SSSR, Vsesoyuz. Inst. Nauchn. i Tekhn.
  Informatsii, Moscow, 1979.

\bibitem{puder_measure_2015}
Doron Puder and Ori Parzanchevski.
\newblock Measure preserving words are primitive.
\newblock {\em J. Amer. Math. Soc.}, 28(1):63--97, 2015.

\bibitem{sela_diophantine_2001}
Z.~Sela.
\newblock Diophantine geometry over groups. {{I}.} {{M}akanin--{R}azborov}
  diagrams.
\newblock {\em Publications Math\'ematiques. Institut de Hautes {\'Etudes}
  Scientifiques}, 93:31--105, 2001.

\bibitem{serre_galois_1997}
{Jean-Pierre} Serre.
\newblock {\em Galois cohomology}.
\newblock {Springer-Verlag}, Berlin, 1997.
\newblock Translated from the French by Patrick Ion and revised by the author.

\bibitem{wilton_halls_2008}
Henry Wilton.
\newblock Hall's {{T}heorem} for limit groups.
\newblock {\em Geometric and Functional Analysis}, 18(1):271--303, 2008.

\bibitem{wilton_essential_2018}
Henry Wilton.
\newblock Essential surfaces in graph pairs.
\newblock {\em J. Amer. Math. Soc.}, 31(4):893--919, 2018.

\bibitem{wilton_distinguishing_2017}
Henry Wilton and Pavel Zalesskii.
\newblock Distinguishing geometries using finite quotients.
\newblock {\em Geom. Topol.}, 21(1):345--384, 2017.

\bibitem{wilton_profinite_2019}
Henry Wilton and Pavel Zalesskii.
\newblock Profinite detection of 3-manifold decompositions.
\newblock {\em Compos. Math.}, 155(2):246--259, 2019.

\bibitem{wise_subgroup_2000}
Daniel~T. Wise.
\newblock Subgroup separability of graphs of free groups with cyclic edge
  groups.
\newblock {\em The Quarterly Journal of Mathematics}, 51(1):107--129, 2000.

\end{thebibliography}

\Addresses

\end{document}